\documentclass[a4paper,12pt]{amsart}
\usepackage[T1]{fontenc} 
\usepackage[all]{xy}
\usepackage[english]{babel}
\usepackage{frcursive}
\usepackage{eucal}
\usepackage{graphicx}
\usepackage{epsfig}
\usepackage{mathrsfs}
\usepackage{amssymb}
\usepackage{amsxtra}
\usepackage{enumerate}
\input cyracc.def

\newtheorem{theorem}{Theorem}[section]
\newtheorem{proposition}[theorem]{ Proposition} 
\newtheorem{lemma}[theorem]{ Lemma}

\newtheorem{definition}[theorem]{Definition} 

\theoremstyle{remark}

\newtheorem{example}{\it Example}

\def \1{\mathbb {1}}
\def \RM{\mathbb {R}}
\def \NM{\mathbb{N}}
\def \ZM{\mathbb{Z}}
\def \CM{\mathbb{C}}

\def \Id {{\rm Id\,}}

\def \d{\partial}
 
\def\a{\alpha}
\def\b{\beta}
\def\e{\varepsilon}  
\def\g{\gamma}

\def\l{\lambda}

\def \s{\sigma}

\def \to{\longrightarrow} 

\def \alg{\mathfrak{g}}

\def \< {{\thengle }}
\def \> {{\rangle }}
\def \( {\left( }
\def \) {\right) }

\newcommand{\Bt}{{\mathcal B}}

\newcommand{\Ht}{{\mathcal H}}

\newcommand{\Lt}{{\mathcal L}}
\newcommand{\Mt}{{\mathcal M}}

\renewcommand{\mod}{{\rm  mod\,}}
\title[AN  ABSTRACT KAM THEOREM]{AN ABSTRACT  KAM THEOREM}
\author{  Mauricio  Garay}
\address{ Institut f\"ur Mathematik\\
FB 08 - Physik, Mathematik und Informatik\\
Johannes Gutenberg-Universit\"at Mainz\\
Staudinger Weg 9\\
55128 Mainz.}
\begin{document}
\begin{abstract} The KAM iterative scheme turns out to be effective in many problems arising in perturbation theory. I propose an abstract version of the KAM theorem to gather these different results.
\end{abstract}
\maketitle
\section*{Introduction}
The KAM iteration scheme was first proposed  in order to prove the stability of invariant tori in Hamiltonian systems. It was
initially conceived, by Kolmogorov, as a Newton type iteration in which each successive step gives a new function defined over some smaller neighbourhood. These neighbourhoods form a decreasing sequence which finally converges to some neighbourhood~\cite{Kolmogorov_KAM}. At each step of the iteration, the space of functions on the given neighbourhood  is  a Banach space. In this way, we get a nested sequence of vector spaces forming a directed system of Banach spaces. 

This KAM process has been adapted to many other cases~(see for instance \cite{Sevryuk_reversible,Stolovitch_KAM}).  The purpose of this paper is to provide a common abstract KAM theorem. It is part of the program I proposed to solve the Herman conjecture~(see~\cite{Oberwolfach_2012}).
 
The first steps in the construction of an abstract perturbation theory were done by Moser,  who suggested a theoretical approach based on infinite dimensional group actions and implicit function theorems. This led Hamilton, Sergeraert and Zehnder to the well-established theory of implicit function theorems in Fr\'echet spaces~\cite{Hamilton_implicit,Moser_pde,Sergeraert,Zehnder_implicit}.  

 These authors already emphasised that the application of implicit function theorems to concrete situation might be quite subtle~: the linearised problem should be solved in a whole neighbourhood and not just in one point, the group should admit some non trivial type of parametrisations etc.  It is therefore tempting to provide a general statement for group actions in infinite dimensional spaces. For locally homogeneous space, the question was tackled by Sergeraert in his thesis~\cite{Sergeraert}~(see also~\cite{groupes}). Zehnder tried to go further but could not get more than a heuristic~(see~\cite[Chapter 5]{Zehnder_implicit}).
 
 In this paper, I will  continue Moser's program by including KAM theory as a study of infinite dimensional group actions in the analytic case.

In many points,  the axiomatisation I will present might seem similar to Zehnder's formalisation of the Nash-Moser  theorem~\cite{Moser_KAM,Nash_imbedding,Zehnder_implicit}. It is in fact different~: unlike Fr\'echet spaces, we consider directed systems of Banach spaces and not inverse ones. In the analytic theory, the direct limit of the system corresponds to the space of germs  of holomorphic functions while the inverse limit of a projective system correspond to the space of holomorphic functions on some fixed open neighbourhood.

This difference between direct and inverse limits can be better understood, if we consider the following Cauchy problem in two variables~:
$$\d_t u=\d_z u,\ u(t=0,\cdot)=u_0. $$
If $u_0$ a convergent power series then so is the solution $  u(t,z)=u_0(t+z) $ to this Cauchy problem. Now, we change our functional space and take the Fr\'echet space of holomorphic functions in some fixed open disk, that is, instead of a direct limit of Banach spaces, we consider an inverse one. 

In general, for fixed $t$, the solution $  u(t,z)=u_0(t+z) $ is no longer holomorphic in the initial disk, so there is no solution to our initial value problem in this Fr\'echet space. Note that, this remark is an essential ingredient in the proof of the abstract Cauchy-Kovalevskaïa theorem~\cite{Baouendi,Nagumo,Nirenberg,Nishida,Ovsyannikov}.

Now, denote respectively by $\CM\{ z \}$ and $ \CM\{ t,z \}$ the algebra of convergent power series in $z$ and in $z,t$. The solutions to our Cauchy problem are given by the exponential mapping
$$\CM\{ z \} \to \CM\{ t,z \},\ u_0 \mapsto e^{t\d_z}u_0=u_0(t+z).$$
So the existence of the solution is related to the existence of an exponential.
More generally, we will prove the existence of an exponential map for some direct limits of Banach spaces. 

In KAM theory the situation is more subtle : we need to consider directed systems depending on two parameters, one of which controls the construction of a Cantor-like set. This leads to the definition of Arnold spaces and gives a conceptual approach to the ultra-violet cutoff technique  used in 1963 by Arnold in his proof of the KAM theorem~\cite{Arnold_KAM}. As we shall see, our abstract KAM theorem turns out to be a consequence of the properties shared by exponential mappings in Arnold spaces. 


  \section{Group actions and normal forms}
Let $G$ be a Lie group acting on a finite dimensional smooth $C^\infty$ manifold $V$. We wish to describe the $G$-orbits in the neighbourhood of a point $a \in V$. This means that we search  a submanifold $W \subset V$ containing $a \in V$ for which the map
$$ G \times W \mapsto V,\ (g,x) \mapsto g \cdot x$$
is locally surjective at $(1,a)$. Such a manifold $W$ is called a {\em transversal} to the $G$-action at $a \in V$. Of course, one may take $W$ to be $V$ itself and this will define a trivial transversal. We are rather interested in finding a minimal $W$, our transversal will be optimal if
$$\dim G+ \dim W=\dim V.$$ 

One way of finding  transversals consists
 in linearising the action in the neighbourhood of $a$. 

The action of $G$ induces an action of its Lie algebra $\alg$ on the tangent space to $V$ at $a$ called the {\em infinitesimal action}. 
For simplicity, let us assume that $V$ is an affine space $a+M$ inside a vector space $E$. In that case, the tangent space to $V$ at $a$ gets identified with $M$ and the infinitesimal action satisfies
$$e^{tv}(a+b)=a+t\,v\cdot a+o(t),\ v \in \alg. $$
We say that a vector subspace $F \subset M$ is a {\em transversal to the $\alg$-action} at $a \in E$ if the map
$$ \alg   \to M/F,\ v \mapsto  \overline{v \cdot a}$$
is surjective.

The following proposition is a consequence of the implicit function theorem
\begin{proposition} Let $G$ be a Lie group acting on an affine subspace $a+M $ of a finite dimensional vector space $E$. If $F$ is a transversal to the $\alg$-action then
it is also a transversal to the $G$-action.
\end{proposition}

 \begin{figure}[ht]
\centerline{\epsfig{figure=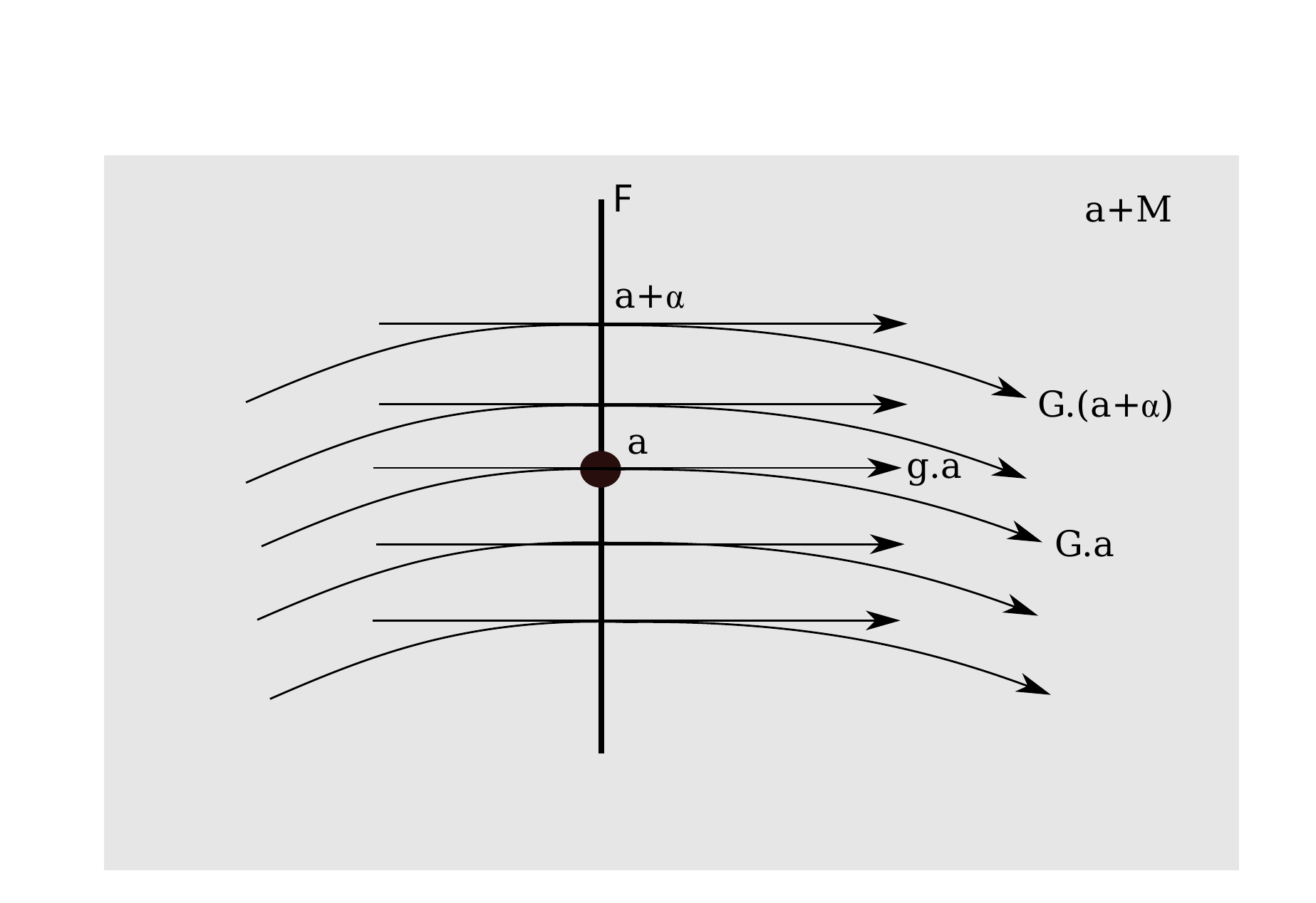,height=0.5\linewidth,width=0.7\linewidth}}
\end{figure} 
The element of a transversal for the linearised action are then called {\em normal forms}. Reduction theory of endomorphisms and the Jordan normal form correspond to the adjoint action of $G=GL(n,\CM)$ on its Lie algebra, so here $E=M=\alg$. If $a$ is diagonal with non-multiple eigenvalues then the $n$-dimensional vector space of diagonal matrices defines a transversal to the adjoint action at $a$. The general case was studied by Arnold in~\cite{Arnold_matrices}

One may also consider infinite dimensional variants. In singularity theory, one studies the case where $G$ is the infinite dimensional automorphism group of a local ring. The corresponding Lie algebra consists of vector fields and the exponential mapping takes a vector field to its flow. One can prove similar propositions, in this setting, called the {\em versal deformation and finite determinacy theorems}~\cite{Mather_fdet,Tyurina}.
 
 Note that it is customary no to use the exponential but rather to associate the automorphism $\Id+v$ to a vector field $v$. These are the type of special parametrisations necessary to apply implicit function theorems in Fr\'echet spaces.
 
In symplectic geometry, this has catastrophic consequences since for a hamiltonian vector field $v$, the automorphism $\Id+v$ will not be symplectic.
The exponential mapping can nevertheless be avoided if we use generating function techniques, as Arnold did in his paper of 1963, but this involves a difficult analysis. Therefore, we shall prefer the exponential parametrisation.

Up to a certain extent the proposition above can be adapted for lagrangian singularities and integrable systems~\cite{lagrange,quantique,VS,VanStraten_Lagrangian}.

But KAM theory is much more difficult in essence. For in this case, the group acts on a directed systems of vector spaces, that we shall call {\em Arnold spaces.}

There is no general theory of infinite dimensional Lie groups in this context, therefore we will consider only actions of closed subgroups of invertible linear mappings. For simplicity, we consider linear actions.

As there are no Lie groups in the infinite dimensional context, there are also no Lie algebras. These are replaced  by some closed vector subspaces of linear mappings. These linear mappings, called {\em $1$-bounded morphisms},  generalise vector fields. Our exponential map is going to be defined from the space of $1$-bounded morphisms to that of invertible morphisms. This process is similar to the one which takes a vector field to its flow.   Let us now proceed to the formal definitions.
\section{Scaled vector spaces}

\begin{definition} An $S$-scaled vector space is a directed system of Banach spaces $E=(E_s)$ indexed by the open interval $]0,S[$ such that the
maps in the directed system are injective with norm at most one.
\end{definition}
 So, in a scaled vector space $(E_s)$ the maps of the directed system 
 $$f_{ts}:E_t \to E_s,\ t>s $$
 are continuous mappings of Banach spaces with $\| f_{ts} \| \leq 1$. As a vector space $E$ is the direct sums of the $E_s$. It is a topological vector space for the product topology.
 
 Sometimes we omit to specify $S$ and write that $E$ is a {\em scaled vector space.}
  
For 
$$x=x_1+\dots+x_n \in {\bigoplus}_{i=1}^n E_{t_i},$$ 
we define
  $$ | x|_s:=\left\{ \begin{matrix} |f_{t_1s}(x_1)+\dots+f_{t_ns}(x_1)|& \ {\rm if\ t_i \geq s },\ \forall i\\  +\infty &\ {\rm otherwise.} \end{matrix} \right. $$
 
There is a trivial $S$-scaled vector space associated to any Banach space $V$ given by the identity mappings 
 $$E_t \stackrel{\Id}{\to} E_s ,\ E_t=E_s=V $$ for $s,t \in ]0,S[$.
 
 A {\em morphism} of $S$-scaled vector spaces 
 $$u=(u_t):E \to F,\ t \in ]0,S[$$  is given by a function
 $$\phi:]0,S[ \to ]0,S[ $$
and a collection of continuous linear mappings  of Banach spaces from $E_t$ to $F_{\phi(t)}$ which commute to the morphisms of the directed system. This means that for any $ t_1 > t_2  >0$, we have a commutative diagram
$$\xymatrix{  E_{t_1} \ar[r] \ar[rd]^{u_{t_1}} & E_{t_2}  \ar[rd]^{u_{t_2}}  \\
  & F_{s_1} \ar[r] & F_{s_2}    }$$ 
with $\ s_i =\phi(t_i)$. Note that we omit to specify $\phi$ in the notation of a morphism. 
 
 This defines the  {\em category of scaled vector spaces}.

We denote by $\Lt(E,F)$ the vector space of morphisms from $E$ to $F$ and when $E=F$, we simply write $\Lt(E)$  for $\Lt(E,E)$.  
The space $\Lt(E,F)$ is endowed with the strong topology.
 
  A scaled vector space $E$ admits many types of decreasing filtrations. The most simple one is given
by the subvector spaces
 $$E^{(k)}=\{ x \in E: \exists C,\tau,\ | x|_s \leq Cs^k,\ \forall s \leq \tau \},\ k \geq 0.$$
 We call it the {\em canonical filtration}. It will be used systematically.
 
 The following two definitions give an abstract form of differential operators.
\begin{definition}[\cite{groupes}]
 Let $E,F$ be $S$-scaled vector spaces. A $\tau$-morphism,  $\tau <S$, 
is a collection of morphisms 
$$(u_\s,\phi_\s):E \to F,\ \s \in ]0,\tau]$$
with $ \phi_\s(t)=t-\s$ and
such that we have a commutative diagram
$$\xymatrix{  E_{s} \ar[d]_{u_{\s_1}}  \ar[rd]^{u_{\s_2}}  \\
   F_{s+\s_1} \ar[r] & F_{s+\s_2}    }$$

for any $s \in ]0,S[,\ \s_1\geq \s_2$.
\end{definition}
For each $\s \in ]0,S[$, we have an evaluation map which assigns to a $\tau$-morphism $(u,\phi)$ the morphism
$(u_\s,\phi_\s)$. Due to the fact that these evaluations are compatible with the maps of the direct system, we abusively write $\tau$-morphisms as morphisms~: $E \stackrel{u}{\to}~F$.
\begin{definition}[\cite{groupes}]
\label{D::borne}
Let $E,F$ be scaled vector spaces.  A  $\tau$-morphism  $ u  $ is called $k$-bounded if
there exists a real number $C>0$ such that~:
  $$| u(x) |_s \leq \frac{C}{(t-s)^k} | x |_{t},\ {\rm for\ any\ }\ s < t \leq \tau,\ x\in E_t  .$$
\end{definition}
For simplicity, we will assume that, for $k=0$, the condition also holds  for $s=t$, so that $u$ maps $E_t$ to $F_t$.  This assumption is unessential but it simplifies some of the notations.
 
A morphism is called {\em $k$-bounded} (resp. {\em bounded}) if there exists $\tau$ (resp. exist $\tau$ and $k$) for which it is a $k$-bounded
 $\tau$-morphism. 

  We denote by $N_\tau^k(u)$ the smallest constant~$C$ which satisfy the estimate
  in Definition \ref{D::borne}  divided by $e=2,71\dots$~:
  $$N_\tau^k(u):= \sup\{ (t-s)^k \frac{|u_n(x)|_s}{e| x |_t} :s < t \leq \tau,x \in E_t \}$$
   It defines  a norm for the space   $\Bt^k_\tau(E,F)$ of $k$-bounded $\tau$-morphisms.
   
 We denote by  $\Bt^k(E,F)$ the vector space of $k$-bounded morphisms  between $E$ and $F$. One easily checks that the normed vector subspaces $(\Bt^k_\tau(E,F),N_\tau^k),\ \tau \in ]0,S[$ define an $S$-scaling of $\Bt^k(E,F)$.

 \begin{proposition}
 \label{P::exp}
 Let $u$ be a $1$-bounded $\tau$-morphism. If the estimate $N_s^1(u) ~<~s $ holds for any $s \leq \tau$ then
 the expansion
 $$e^u:=\sum_{j \geq 0}\frac{u^j}{j!} $$
 converges to a morphism in $\Lt(E)$. Moreover, if
 $$\nu:=\frac{N^1_\tau(u)}{(\tau-s)} \leq \frac{1}{2} $$
  one has the estimates
   \begin{enumerate}[{\rm 1)}]
    \item $\displaystyle{| e^u x |_{s} \leq 2   |x |_\tau}$ ;
 \item $\displaystyle{| (e^{-u}(\Id +u)-\Id)  x|_s \leq   4|x|_\tau \nu^2}$ ;
  \item $\displaystyle{| (e^{-u}-\Id)  x|_s \leq  2 |x|_\tau \nu} $.

\end{enumerate} 
\end{proposition}
\begin{proof}
Take  $x \in E_t$ and choose $s<t$.   As $u$ is $1$-bounded by cutting the interval $[s,t]$ into $n$ equal pieces, we get that~:
 $$| u^n(x) |_{s} \leq \frac{n^n}{e^n(t-s)^n} \left( N_{t}^1(u)\right)^n | x |_t. $$
Using Stirling formula, we get that
$$\lim \left( \frac{n^n}{n!} \right)^{1/n} = e,$$
thus~:
$$\frac{N_t^n( u^n )}{n!} \leq  N_t^1(u)^n. $$
The previous estimate shows that
$$| e^u x |_{s} \leq \sum_{j \geq 0} \frac{(N^1_t(u))^j }{(t-s)^j} | x |_t=\frac{1}{1-\nu }   |x |_t.$$
This proves the convergence of the  exponential.
As for $\nu \leq 1/2$, we have $1/(1-\nu) \leq 2$. This proves the first estimates 
by taking $t=\tau$.

Let us now prove the second estimate. The expansion
$$ e^{-u}(\Id +u)-\Id=\sum_{n \geq 0} \frac{(n+1)}{(n+2)!}(-1)^{n+1}u^{n+2}$$
gives
$$ \left| \sum_{n \geq 0} (-1)^{n+1} \frac{(n+1)}{(n+2)!}u^{n+2} ( x) \right|_s \leq  | x |_\tau \sum_{n \geq 0}  
\frac{(n+1)}{(\tau-s)^{n+2}}N^1_\tau(u)^{n+2}=\frac{\nu^2}{(1-\nu)^2}| x |_\tau .$$
As
$$\frac{1}{(1-\nu)^2} \leq 4,\ \forall \nu \in \left[0,\frac{1}{2}\right], $$
we get the estimate 2). The proof of 3) is similar, we leave it to the reader.
\end{proof}

 Note that, by definition, for $u \in \Bt^k(E)^{(2)}$, there exists $C>0$ such that
 $$N_s^k(u) \leq Cs^2. $$ 
for any $s$ small enough. In this case, the conditions of Proposition~\ref{P::exp} are automatically satisfied for $\tau$ small enough.

\section{Arnold spaces}         

\begin{definition} An $S$ pre-Arnold space $E_\bullet$  is  a  product of $S$-scaled vector spaces indexed by $\overline{\NM}:=\NM \cup \{ +\infty \}$~:
$$E_\bullet:=\prod_{n \in \overline{\NM}} E_n.  $$
\end{definition}
As for scaled vector spaces, we sometimes omit the index $S$. We endow pre-Arnold spaces of the direct product topology.

A map of pre-Arnold spaces  $u_\bullet:E_\bullet \to F_\bullet$
is a {\em morphism} if 
 \begin{enumerate}[{\rm i)}]
\item $u_\bullet(E_n) \subset~F_n$, $\forall n \in \NM$ ;
\item the map $u_\bullet$ induces morphisms of scaled vector spaces $u_n:E_n \to F_n$.
\end{enumerate}
 This defines the {\em category of pre-Arnold spaces.}  In Arnold spaces, $\tau$-morphisms and $k$-bounded $\tau$-morphism are defined componentwise. We define the {\em norm} of a  $k$-bounded $\tau$-morphism 
$$u_\bullet:E \to F $$
as the sequence 
$$N^k_\tau(u_\bullet):=(N^k_\tau(u_n)). $$
So it is not a norm in the usual sense of it, but rather a sequence of norms.
 
 The filtrations defined for scaled vector spaces extend naturally to  pre-Arnold spaces, for instance~:
$$x_{\bullet} \in E^{(k)}_\bullet \iff x_n \in (E_n)^{(k)},\ \forall n \in \NM.$$

\begin{definition} An $S$-Arnold space $E_\bullet$ (or simply an Arnold space) is a directed system of $S$-scaled vector spaces indexed by  $\NM \cup \{ \infty \}$ and such that the morphisms of the directed system
$$r_{ij}:E_i \to E_j,\ i,j \in \NM \cup \{ \infty \},\ i <j $$
are $0$-bounded with norm at most one.
\end{definition}
Note that $E_\infty$ is   the limit of the directed system. 

The maps $r_{nm}$ are  called {\em restriction morphisms}. For simplicity, they are assumed to be $0$-bounded, this condition can be relaxed by $k$-bounded for arbitrary $k \geq 0$. 

The {\em category of Arnold spaces} is the full subcategory of pre-Arnold spaces having for objects Arnold spaces. In particular, a morphism of Arnold spaces does NOT necessarily commute to restriction mappings.

\begin{example} Let $K_n \subset \RM^d $ be a decreasing sequence of compact subsets and put $K_\infty=\bigcap_n K_n$. Consider the Banach space $C^0(K_n,\RM)$ together with the trivial scaling. The product space 
$$E_\bullet:=\prod_{n \in\NM \cup \{ \infty\}}E_n,\ E_n:= C^0(K_n,\RM)$$ is a  pre-Arnold spaces space.
The maps
$$C^0(K_{i},\RM)  \to C^0(K_{i+j},\RM) , f \mapsto f_{|K_{i+j}} $$
induce restriction mappings $r_{i\, i+j}$. This induces an Arnold space structure on $E_\bullet$. 
\end{example}

Here are some conventions to simplify the notations~:
\begin{enumerate}[{\rm i)}]
\item we use the notation $r_m$  for the restriction map from $E_n$ to $E_m$ instead of $r_{n m}$ ;
\item we use the notation $r$  for the restriction map from $E_n$ to $E_\infty$ instead of $r_{n \infty}$ ;
\item for $x \in E_\bullet$ which projects to $(E_n)_s$, we denote by $| x |_{n,s}$ the norm of its projection ;
\item given morphisms $u:E_n \to E_n$, $v:E_{n+k} \to E_{n+k}$ we write $vu$ for the composed map $ v\,r_{n+k}\,u$.
\end{enumerate}
 
Note that there is a natural functor from the category of scaled vector spaces to that of Arnold spaces by taking the product with itself 
 $$F:SVS \to AS,\ E \mapsto E_\bullet:=E\times E \times E \times \cdots $$
 Restriction mappings are simply identity mappings.
 
In particular any linear group action on a scaled vector space might be seen as an action on the associated Arnold space. In concrete example,
this explains why diophantine conditions can  sometimes be relaxed to Bruno type conditions in small denominator problems. For instance, if we use this functor to prove Siegel's linearisation theorem for vector fields at singular points, then one can show that the diophantine condition is replaced by Bruno's condition~\cite{Brjuno,Siegel_vecteurs}. If we use it in Kolmogorov's invariant tori theorem then we end up with the {\em tameness condition} which will be explained in the sequel.
\section{Convergence of infinite products}
We now arrive at our first important theorem. The theorem that we shall now prove gives a criterion for proving the convergence for
the KAM iterative process. It shows the relevance of exponential maps in scaled vector spaces.
\begin{theorem}
\label{T::produits}
Let $u_\bullet:E_\bullet \to E_\bullet$ be a $1$-bounded $\tau $-morphism of an Arnold space. Assume that there exists $\l \in ]0,1[$ such that
the norm of $u$ satisfies the estimate 
$$\sum_{n \geq 0} N^1_{s}(u_n)\leq (1-\l)s$$
for any $s \leq \tau$, then the sequence $(r\,g_n)$ defined by 
$$g_n=e^{u_n}  \cdots  e^{u_1}  e^{u_0}$$
 converges in  $\Lt(E_0,E_\infty)$ and maps continuously the Banach space $(E_0)_s$ to $(E_\infty)_{\rho s}$ with $\rho=\l(1-\l)$.
\end{theorem}
\begin{proof}
 To prove the theorem, we start with the
 \begin{lemma} 
\label{L::produits}
Let $u_\bullet$ be a $1$-bounded $\tau$-morphism such that $N_s^1(u_i) ~<~s $ for any $i \geq 0$.
For any $s \leq \tau$ and any $x \in (E_0)_s$, we have the estimate
$$| g_n x |_{n,\l s} \leq   \frac{1}{1-\frac{1}{(1-\l)s}\sum_{i \geq 0} N_s^1(u_i)} | x |_{s,0} $$
provided that $\l$ is such that
$$\max_{i \leq n}\frac{1}{(1-\l)s}N_s^1(u_i)<1. $$
 \end{lemma}
\begin{proof}
Denote by $C_{j,n} \subset \ZM^j$, the multi-indices $ I=(i_1,\dots,i_j)$ with coordinates in
 $\{0,\dots,n \}$.  We have the formula
 $$ \frac{1}{1-\a(\sum_{k=0}^n z_k)}=\sum_{j \geq 0} \a^j \sum_{I \in C_{j,n}}  z_I,\ z_I:=z_{i_1}z_{i_2}\cdots z_{i_j}$$
 for any $\a \in \RM$, $I=(i_1,i_2,\dots,i_n)$.

For $I=(i_1,\dots,i_j) \in C_{j,n}$, we denote by $\s(I)$ the vector whose components are obtained by permuting those of $I$ in such a way that
$\s(I)_p \geq \s(I)_{p+1}$. 

We put
$$u[I]:= u_{\s(I)_1} u_{\s(I)_2} \cdots u_{\s(I)_j},\ I \in C_{j,n}. $$
We expand $g_n$ and collect the terms in the following way~:
$$g_n=\sum_{j \geq 0}\frac{1}{j!} (\sum_{I \in C_{j,n}}  u[I]) =1+\sum_{i=0}^n u_i+\frac{1}{2}(\sum_{i=0}^n u_i^2
+\sum_{j=0}^n \sum_{i=j+1}^n 2 u_i u_j) +\dots .$$

Fix $s $ and put 
$$z_I:=N_s^1(u_{i_1})N_s^1(u_{i_2})\cdots N_s^1(u_{i_n}).$$
The inequality
$$\frac{1}{j!}N_s^j( u[I])  \leq \prod_{p=0}^jN_s^1(u_{i_p})=  z_I$$
implies that
$$\left| u[I] (x)\right|_{n,\l s} \leq   \left(\frac{1}{(1-\l)s} \right)^j z_I  | x|_{s,0},\ \forall \l \in ]0,1[. $$
Put
$$ \a=\frac{1}{(1-\l)s},$$
we get the estimate
$$| g_n x |_{n,\l s} \leq  \left( \sum_{j \geq 0}\a^j \sum_{I \in C_{j,n}}  z_I \right)| x|_{s,0}=\frac{1}{1-\a(\sum_{k=0}^n z_k)}| x|_{s,0}  .  $$
This proves the lemma.
 \end{proof}
  
  Fix  $s \in ]0,\tau]$ and denote by   $\| \cdot \|_{\a,n}$ the operator norm in the vector space $\Lt( (E_0)_s,(E_n)_{\a s})$.   
  
 The preceding lemma gives the estimate
 $$ \| g_n \|_{\l,n} \leq C,\  C:= \frac{1}{1-\frac{1}{(1-\l)s}\sum_{i \geq 0} N_s^1(u_i)}.$$
 Thus the sequence $(g_n)$ defines by a uniformly bounded sequence in $\Lt((E_0)_s,(E_\infty)_{\l s})$.  
 
   We have
   $$\sup_{i \geq 0} N_{\l s}^1(u_i) <\rho s ,\ \rho:=(1-\l)\l .$$ 
  
 I assert that the sequence with elements $ \| g_n-r_ng_{n-1}\|_{\rho,n }$ is convergent.  To see it write
$$g_n-r_n g_{n-1}=(e^{u_n}-\Id)r_n g_{n-1} $$
where $\Id_n \in \Lt(E_n)$ denotes the identity mapping.

By Proposition~\ref{P::exp}, we get that~:
$$ | (e^{ u_n}-\Id_n) y |_{n,\rho s} \leq  \left( \sum_{j \geq 0} \frac{ (  N_{\l s}^1(u_n))^{j+1} }{(\rho s )^{j+1} } \right) | y |_{n,\l s}=
\frac{1}{1-\mu-\frac{N^1_{\l s}(u_n)}{\l s}} \frac{N^1_{\l s}(u_n)}{\l s} | y |_{n,\l s} .$$
 Take $y=g_n x$,  this gives the estimate
$$ \| (e^{ u_n}-\Id_n) g_{n-1} \|_{\rho,n} \leq   \frac{ C}{1-\l-\frac{N^1_{\l s}(u_n)}{\l s}} \frac{N^1_{\l s}(u_n)}{\l s}.$$
We proved that
 $$ \| g_n-r_ng_{n-1}\|_{\rho,n} \leq  K \frac{N^1_{\l s}(u_n)}{\l s}$$
with
$$K:=\sup_{n \geq 0} \frac{ C}{1-\l-\frac{N^1_{\l s}(u_n)}{\l s}}. $$
Thus the sequence $(rg_n),\ n \geq j$ defines a Cauchy sequence in the Banach space $\Lt((E_0)_s,(E_\infty)_{\rho s})$.  
 This proves the theorem.
 \end{proof}

 \section{The abstract KAM theorem}
 The notions of morphisms and bounded morphisms extend to pre-Arnold spaces but these are not be sufficient to prove the convergence of our iterative procedure of the abstract KAM theorem. So we need a third notion which generalises the diophantine condition of the KAM theorem~: tameness.
 
 Unlike the diophantine condition  which  can be seen as a notion for scaled vector spaces, tameness is specific to morphisms between  Arnold spaces.  
  
  We say that an increasing real positive sequence $p:=(p_n)$ is {\em tamed} if
  $$\sum_{n \geq 0} \frac{\log p_n'}{2^n}<+\infty,\ p_n':=\max(1,p_n). $$  
  For instance, the sequence $(e^{A^n})$ is tamed provided that $A<2$. This condition was introduced by Bruno in the context of diophantine approximation for linearising vector fields~\cite{Brjuno}.  
  
 \begin{definition}
 A  $k$-bounded $\tau$-morphism between pre-Arnold spaces
 $$u_\bullet=(u_n):E_\bullet \to F_\bullet $$
 is called $k$-tamed if the sequence  $N^k_{\tau}(u_\bullet)=(N^k_{\tau}(u_n))$ is tamed.
 \end{definition}
We denote by $\Mt^k(E_\bullet,F_\bullet)$ the vector space of $k$-tamed morphisms and for $E_\bullet=F_\bullet$, we write $\Mt^k(E_\bullet)$ instead of $\Mt^k(E_\bullet,E_\bullet)$. It is a vector subspace of $k$-bounded morphisms, therefore it admits a natural filtration 
$$\Mt^k(E_\bullet,F_\bullet)\subset \Mt^k(E_\bullet,F_\bullet)^{(1)} \subset \Mt^k(E_\bullet,F_\bullet)^{(2)} \subset \cdots $$ 
Observe that we have an inclusion of vector spaces
$$\Mt^k(E_\bullet,F_\bullet) \subset \Bt^k(E_\bullet,F_\bullet) \subset \Lt(E_\bullet,F_\bullet) $$
for any $k \geq 0$.

We now give a nonlinear definition of tameness.   Given a sequence $r=(r_n)$, we define the $k$-ball centred at the origin in $E_\bullet$ of radius $r$ by putting
$$B(r,k)=\bigcup_t B(r,k)_{n,t},\ B(r,k)_{n,t}=\{ x \in E_{n,t}: | x |_{n,t} < r_n t^k \}. $$
\begin{definition} Let $r$ be a real sequence and $k$ a positive integer. A map of Arnold spaces
$f:E_\bullet \supset B(r,k) \to F_\bullet$ is tamed by a sequence $(p_n)$, $p_n^{-1}<r_n$, if there exists $C>0$ such that
 $$ | x |_{n,t} \leq \frac{(t-s)^k}{p_n} \implies | f(x)|_{n,s} \leq C$$
 and $(p_n)$ is tamed.
  \end{definition}
A tamed (linear) morphism is, of course, a tamed mapping.  

If a map is tamed by a sequence, then this sequence is not unique. It
can nevertheless be chosen in the following appropriate way~:
\begin{proposition} 
\label{P::tamed} Let $r$ be a real sequence and $k$ a positive integer and 
$f:E_\bullet \supset B(r,k) \to F_\bullet$ a map of Arnold spaces tamed by a sequence $(p_n)$. For any $C \geq 1$ and any $A \in ]1,2[$, the map is tamed
by a sequence $(q_n)$ such that
\begin{enumerate}[{\rm i)}]
\item $q_n \geq e^{A^n}, \forall n \in \NM$ ;
\item $q_n^2 \geq C q_{n+1}, \forall n \in \NM$ ;
\item there exists $N$ such that $q_n=p_n,\ \forall n>N$.
\end{enumerate}
\end{proposition}
\begin{proof}
The proof is elementary. For instance, to get i) just replace $(p_n)$ by
$$\max(e^{A^n},p_n). $$
To get ii) and iii), note that if $(p_n)$ is tamed then there exists $N$ such that
$$p_n^2 \geq C p_{n+1},\ \forall n \geq N .$$
We replace $(p_n)$ by
$$p'_n=C\max_{i \leq N} p_i, \forall n < N  $$
and $p'_n=p_n$ for $n \geq N$. This proves the proposition.
\end{proof}

We now define approximated inverses of linear mappings in Arnold spaces. To do this, we first introduce a new filtration on a scaled vector space~:
 the {\em harmonic filtration}. The filtration depends on the choice of $d \geq 0$. For a given scaled vector space $E$, the terms of the filtration are defined by
$$ \Ht_d^k(E_t)=\{ x \in E_t:    | x |_s  \leq \frac{1}{(t-s)^d} \left(\frac{s}{t}\right)^{2^k}| x |_t,\ \forall s<t   \}.$$   
In the classical KAM theorem, it is the filtration with respect to the degrees of the harmonics appearing in the Fourier expansion of a function.

We may now proceed to the definition of a quasi-inverse~:
\begin{definition} Let $u:E_\bullet \to F_\bullet$ be a morphism of Arnold spaces. A right $d$-quasi inverse of a morphism $u:E_\bullet \to F_\bullet$  is a morphism $v:F_\bullet \to E_\bullet$ such that
$$ y-uv(y) \in \Ht^{\,n}_d(F_n),\ {\rm for\ any\ } n \in \NM,y \in F_n .$$
\end{definition}

There is a natural notion of bounded splitting in an Arnold space that we shall also use in our formulation of the theorem~:
if $E_\bullet$ is an Arnold space, we say that $F_\bullet,G_\bullet$ define a {\em $k$-bounded splitting of $E_\bullet$} if
\begin{enumerate}[{\rm i)}]
\item $E_\bullet$ is the direct sum of $F_\bullet$ and $G_\bullet$ ;
\item  the subspaces are $F_\bullet$ and $G_\bullet$ are closed ;
\item the projections on each factor is $k$-bounded with norm at most one.
\end{enumerate}
\begin{theorem}
 \label{T::KAM} Let $E_\bullet$ be an Arnold space, $M_\bullet $ a closed subspace of $E_\bullet^{(\mu)}$ for some $\mu \geq 0$ and
 $F_\bullet,G_\bullet$ an $m$-bounded splitting of $M_\bullet $, for some $m \geq 0$. Consider a vector subspace $\alg_\bullet \subset \Mt^1(E_\bullet)^{(2)}$ which leaves   $G_\bullet$ invariant and let $a \in E_0$ be such that $\alg_\bullet$ maps $a$ in $G_\bullet$.  We consider the linear maps
 $$\rho(\a): \alg_\bullet \to G_\bullet,  u \mapsto u (a+\a) .$$
Assume that for some $k,l \geq 0$ and some sequence $r$~:
 \begin{enumerate}[{\rm A)}]
\item for each $\a \in F_\bullet$, there is a $k$-tamed right quasi-inverse $j(\a) \in \Mt^k( G_\bullet,\alg_\bullet)$ of $\rho(\a)$ ;
\item the map   $j:F_\bullet \supset B(r,l) \to \Mt^k( G_\bullet,\alg_\bullet),\ \a \mapsto j(\a)$ is $l$-tamed ;
\item $\mu > 2k+l+m+2  $
\end{enumerate}
then for any $b \in M_0 $ and any $A \in ]1,2[$ , there exists  a $1$-tamed morphism $u_\bullet=(u_n)$ of $E_\bullet$ such that
 \begin{enumerate}[{\rm i)}]
 \item $u_i \in \alg_i$ ;
 \item   $N^1_s(u_n)< s^{k+l+m+2}e^{-A^n} $ for any $s$ sufficiently small and any $n \geq 0$~;
 \item $  g(a+b)=r(a) (\mod F_\infty) $ where   $g$ is the limit of the sequence $(e^{u_n} e^{u_{n-1}} \dots  e^{u_1}  e^{u_0})_{n \in \NM} \subset \Lt(E_0,E_\infty)$.  \end{enumerate}
 \end{theorem}
 \section{Proof of the abstract KAM theorem}
We use an iteration similar to the one introduced by Kolmogorov and modified by Arnold in the proof of the KAM theorem~\cite{Arnold_KAM,Kolmogorov_KAM}.
 
 \subsection*{The iterative process}
 The iteration depends on the choice of  some appropriate integer $N$, that will be fixed later.
We define
$$\phi:\NM \to \NM $$
by
$$\phi(n)=\left\{\begin{matrix} N & {\rm for \ }& n \leq N \\
                                                 n & {\rm for \ }& n > N \end{matrix} \right. $$
 Denote respectively by $\pi_F:M \to F ,\ \pi_G:M \to G $ the projections to $F$ and $G$.
                                                
 We define inductively the sequences $(\b_n)$, $(u_n)$ by putting
$$ \b_0=b,\ u_0=j(0)(b) ;$$
$$ \left\{ \begin{matrix}
 \b_{n+1}&=&e^{-u_n} (a_n+\b_n)-a_{n+1} ;\\
  u_{n+1}&=&j(\sum_{i=0}^n\a_i)\pi_G( \b_{n+1}) \end{matrix} \right.$$ 
 where
$$\left\{ \begin{matrix}   \a_{n}&=&\pi_F(u_n(a_n)-\b_n)&=&\pi_F(-\b_n) ; \\
\g_n&=&\pi_G(u_n(a_n)-\b_n)\\  a_{n+1}&=& a_n+\a_n .\end{matrix} \right.  $$  
 I assert that if $(u_n)$ tends exponentially fast to zero as stated in part i) of the theorem and if $(\b_n)$ tend to zero then this implies ii) and concludes the proof of the theorem.
 Indeed,  by Theorem~\ref{T::produits} the sequence
$$g_n:=e^{u_n}   e^{u_{n-1}} \dots  e^{u_1}   e^{u_0}$$
converges to a morphism $g$. By  taking the image under $r$  in the equality
$$a_{n+1}+\b_{n+1}=r_{n+1}g_n (a+b) $$ 
and passing to the limit, we get that~: 
$$r(a)=g (a+b)\ (\mod F_\infty).$$ 
This proves assertion.

\subsection*{Choice of a subdirected system}
Let $E_\bullet$ be an Arnold space. Any decreasing real positive sequence $(s_n)$ defines a doubly directed system of Banach spaces
$$ \xymatrix{(E_0)_{s_0} \ar[r] \ar[d]  & (E_0)_{s_1} \ar[r] \ar[d] &  (E_0)_{s_2} \ar[r] \ar[d]& \cdots \\
 (E_1)_{s_0} \ar[r] \ar[d]& (E_1)_{s_1} \ar[r] \ar[d]&  (E_1)_{s_2} \ar[r] \ar[d]& \cdots \\
 (E_2)_{s_0} \ar[r] \ar[d]& (E_2)_{s_1} \ar[r] \ar[d]&  (E_2)_{s_2} \ar[r] \ar[d] & \cdots \\
 \cdots & \cdots & \cdots & \cdots
}$$
We shall now construct a sequence $(s_n)$ having a non zero limit $s_\infty$. Then, we will  interpret the sequence $ (\b_n),\ n \in \overline{\NM}$, with $\b_\infty=0$, as an element of the Arnold space which belongs to the directed system $(E_n)_{s_n},\ n \in \overline{\NM}$.

Let $(p_n)$ be a sequence for which the map $j$ is tamed bounded by $C>1$. According to
Proposition~\ref{P::tamed}, without loosing any generality, we may assume that $(p_n)$ 
$$p_n \geq e^{A^n},\ p_n^2 \geq C^2p_{n+1}2^{4k+2l+2m+5} $$
 for any $n \in \NM $.

Next, we consider the sequence $(\rho_n)$ defined by~:
$$\rho_n:= p_{n+1}^{-1/2^n}. $$
As $(p_n)$ satisfies Bruno's condition, the infinite product 
$\prod_{n \geq 0} \rho_n   $  is strictly positive.

\begin{lemma} There exists a constant $C' \geq 1$ such that
$$1-\rho_n  \geq \frac{1}{C'e^22^n} $$
\end{lemma}
\begin{proof}
As $(p_n)$ is tamed, there exists a constant $C' \geq 1$ such that
$$p_n \leq C'e^{2^n}. $$
We have
$$1-\rho_n    = 1-p_{n+1}^{-1/2^n}=1-e^{ -\frac{\log p_{n+1}}{2^n}}.  $$
Now
$$1-e^{-X} \geq \frac{X}{a}\ {\rm for\ }X \in [0,\log a], $$
for $a>1$. Moreover
$$ \frac{\log p_{n+1}}{2^n} \leq \frac{\log C'}{2^n}+2 \leq \log C'+2$$
thus
 $$1-(p_{n+1})^{-1/2^n}=1-e^{ -\frac{\log p_{n+1}}{2^n}} > \frac{\log p_{n+1}}{C'e^22^n}.$$
By definition of $(p_n)$, we have
 $$  \frac{\log p_{n+1}}{2^n} \geq   \frac{A^{n+1}}{2^n} > \frac{1}{2^n}  . $$
 This proves the lemma.
 \end{proof}
 
We now define the sequence $(s_n)$  by~:
$$ s_{n+1}:=\rho_n^5 s_n. $$
and put
$$\lim_{n\to +\infty}s_n=s_\infty=\left(\prod_{n \geq 0} \rho_n \right)^5 s_0>0.$$

We now define the sequence $(\s_n) $ defined by  
$$ \s_n:=\frac{s_\infty}{C'e^2 2^n}=\frac{\left(\prod_{n \geq 0} \rho_n \right)^5 s_0 }{C'e^2 2^n} ,$$
the initial term $s_0$ will be fixed below.
The previous lemma shows that~:
$$s-\rho_n s \geq \s_n $$
provided that $\rho_n s \geq s_{\infty}$.
\subsection*{Initialisation of the induction}
For $s$ sufficiently small, we denote  $| \a_n |_s $ the norm of $\a_n$ in $(E_{n})_s$ instead of $| \a_n |_{n,s} $ and similarly for the other sequences involved in the iteration.  By multiplication of all norms by a constant, we may also assume that~: 
  $$| a |_{s_0} \leq \frac{1}{4}. $$
  This will simplify our estimates when applying Proposition~\ref{P::exp}.

Define the sequence $(\e_n)$ by putting
$$\e_n=p_n^{-1} \s_n^{2k+l+m+2}. $$
 Note that the estimate 
 $$p_n^2 \geq  C^2 p_{n+1}2^{4k+2l+2m+4}  $$
 implies that
 $$C^2\e_n^2   \leq  \e_{n+1} \s_{n+1}^{2k+l+m+2}$$
The sequences $(\rho_n),\ (\s_n)$ defined above depend on the choice of $s_0$. According to condition C), we may choose $s_0<1$ small
enough such that the estimates i), ii)  below hold for $n=0 $~:
  \begin{enumerate}[{\rm i)}]
  \item $\displaystyle{| \a_n|_{\rho_n^4 s_n} \leq    \e_n }$ ;
 \item $\displaystyle{| \b_n|_{s_n} \leq    \e_n  }$ ;
  \item $\displaystyle{| \g_n |_{\rho_n^4 s_n} \leq \frac{1}{6}  \e_{n+1}\s_{n+1}^m }$.
 \end{enumerate}

We choose $N$ in the definition of $\phi$ such that
$$ p_n^{-2^{\phi(n)-n}}
 \leq  \s_{n+1}^{d+m} \e_n.$$
 for any $n \in \NM$. With this choice, we have the~:

\begin{lemma} 
\label{L::cutoff}
For  any $\g \in \Ht^{\phi(n)}_d(E_n)_{s}$, we have the estimate
 $$ | \g |_{\rho_n s} \leq   \s_{n+1}^{m} \e_n | \g |_s $$
 provided that $\rho_n s \geq s_\infty$.
\end{lemma}
\begin{proof}
As $\g \in \Ht^{\phi(n)}_d(E_n)_{s}$, we have that
 $$ | \g |_{\rho_n s} \leq   \frac{1}{\s_n^d}\rho_n^{2^{\phi(n)}} | \g |_s.$$
The definition of $(\rho_n)$ implies that
$$ \frac{1}{\s_n^d}\rho_n^{2^{\phi(n)}}  \leq \frac{1}{\s_n^d} p_n^{-2^{\phi(n)-n}} \leq  \s_{n+1}^{m}\e_n.$$
This concludes the proof of the lemma.
\end{proof}

To conclude the proof of the theorem, we need to show the validity of the estimates i), ii) and iii). We will do this by induction.  So, we assume that i), ii) and iii) are proved up to index $n$.   
 
\subsection*{Estimate for $u_n$} 
We put $\a_{-1}=0$.  
We have
As $p_i^{-1} \leq p_{n}^{-1}$ for $i \leq n$, using i), we get that
 $$\sum_{i=-1}^{n-1}| \a_i |_{\rho_i s_i} \leq p_n^{-1} \sum_{i= 0}^{n-1}\s_i^{2k+l+m+2} \leq p_n^{-1}s_\infty^l.$$
  
As the map $j$ is $l$-tamed by the sequence $(p_n)$, we get the estimate~:
  $$N_{\rho_{n-1}^2 s_{n-1}}^k (j_n(\sum_{i=-1}^{n-1}\a_i)) \leq  C   .$$

Thus $u_n=j(\sum_{i=0}^{n-1}\a_i)( \b_n)$ satisfies the inequality
$$ N^1_{\rho_n s_n}(u_n) \leq  \frac{C}{\s_n^k} | \b_n |_{ s_n} .$$
Combined with ii), we get that~:
 $$(*)\ N^1_{\rho_n s_n}(u_n) \leq   C\e_n\s_{n}^{-k}.$$

We now write  $\b_{n+1}$ as
$$\b_{n+1}=A_n+B_n+C_n, $$
with
$$\left\{ \begin{matrix} A_n&:=&(e^{-u_n}(\Id +u_n)-\Id) a_n\ ;
\\ B_n&:= &(e^{-u_n}-\Id)\a_n\ ;\\
 C_n&:=&e^{-u_n}  \g_n .\end{matrix} \right.$$  

\subsection*{ Estimate for $A_n$} 
By Proposition~\ref{P::exp}, we get the estimate~:
$$| A_n|_{\rho_n^2 s_n} \leq  \frac{ N^1_{\rho_n s_n}(u_n)^2}{\s_n^2}$$
Using $(*)$ at index $n$,  we get that~:
$$| A_n|_{\rho_n^2 s_n} \leq   C^2\e_n^2\s_{n}^{-2k-2}$$
and
$$ C^2\e_n^2 \s_n^{-2k-2}  \leq  \frac{1}{3}  \e_{n+1}\s_{n+1}^m   .$$
 
\subsection*{Estimate for $B_n$} 
 Applying again Proposition~\ref{P::exp}, we get the estimate~:
$$| B_n|_{s_{n+1}} \leq  2\frac{ N_{\rho_n s_n}(u_n)}{\s_n} |\a_n|_{\rho_n^4 s_n}.$$
The induction hypothesis and $(*)$ at index $n$ then gives~:
$$| B_n|_{s_{n+1}} \leq  2 C \e_n^2\s_n^{-k-1}$$
and
$$ 2 C \e_n^2  \s_{n}^{-k-1} \leq \frac{1}{3}\e_{n+1} \s_{n+1}^m  . $$

\subsection*{ Estimate for $C_n$} 
 Proposition~\ref{P::exp}  gives the estimate~:
$$ |C_n|_{s_{n+1}} \leq   2 |\g_n|_{\rho_n^4 s_n}  \leq \frac{1}{3}   \e_{n+1} \s_{n+1}^m.$$
 
\subsection*{Estimate for $\b_{n+1}$} 
The three estimates for $A_n,B_n$ and $C_n$ imply that
$$| \b_{n+1}|_{s_{n+1}} \leq   \e_{n+1} \s_{n+1}^m. $$
This proves ii) at index $n+1$.

\subsection*{Estimate for $\a_{n+1}$} 
As the projections $\pi_F$ and $\pi_G$ are $m$-bounded with norm at most one,  we get using the estimates for $A_n,B_n,C_n$~:
$$ |\pi_F(\b_{n+1}) |_{\rho_{n+1} s_{n+1}} \leq  \e_{n+1}. $$
and similarly for $\pi_G$. As $\a_{n+1}=\pi_F(\b_{n+1})$ this proves i) at index $(n+1)$.
 
  \subsection*{ Estimate for $ \g_{n+1}$} 
  The map $\pi_G$ is  $m$-bounded with norm at most one and 
 $$  \g_{n+1}= \pi_G(\b_{n+1})-\pi_G(u_{n+1}(a_{n+1})).  $$
 We have
$$| \pi_G(\b_{n+1}) |_{\rho_{n+1} s_{n+1}} \leq  \e_{n+1}. $$
We now estimate the second term in the right-hand side~:
 $$|\pi_G(u_{n+1}(a_{n+1}))|_{\rho_{n+1}^3 s_{n+1}} \leq \s_{n+1}^{-m} |u_{n+1}(a_{n+1})|_{\rho_{n+1}^2 s_{n+1}} $$ 
  Using $(*)$ at index $(n+1)$, we get the estimate
 $$|u_{n+1}(a_{n+1})|_{\rho_{n+1}^2 s_{n+1}} \leq  \e_{n+1}  \s_{n+1}^{-m-k}$$
 
Combining these estimates with the ones for $\b_n$, we get that~:
 $$ |\g_{n+1} |_{\rho_{n+1}^3 s_{n+1}} \leq   (1+ \s_{n+1}^{-m-k}) \e_{n+1}.$$
 Now, we use Lemma \ref{L::cutoff} and get the inequality~:
$$ |\g_{n+1} |_{\rho_{n+1}^4 s_{n+1}} \leq \s_{n+2}^{m}\e_{n+1}   |\g_{n+1} |_{\rho_{n+1}^3 s_{n+1}}
 \leq  (1+ \s_{n+1}^{-m-k})\s_{n+2}^{m} \e_{n+1}^2 \leq \frac{1}{6} \s_{n+2}^m\e_{n+2}.$$
 This proves iii) at index $ n+1$ and concludes the proof of the theorem.
 \bibliographystyle{amsplain}
\bibliography{master}

\providecommand{\bysame}{\leavevmode\hbox to3em{\hrulefill}\thinspace}
\providecommand{\MR}{\relax\ifhmode\unskip\space\fi MR }
\providecommand{\MRhref}[2]{%
  \href{http://www.ams.org/mathscinet-getitem?mr=#1}{#2}
}
\providecommand{\href}[2]{#2}
\begin{thebibliography}{10}

\bibitem{Arnold_KAM}
V.I. Arnold, \emph{{ Proof of a theorem of A. N. Kolmogorov on the preservation
  of conditionally periodic motions under a small perturbation of the
  hamiltonian}}, Uspehi Mat. Nauk \textbf{18} (1963), no.~5, 13--40, English
  translation: Russian Math. Surveys.

\bibitem{Arnold_matrices}
\bysame, \emph{On matrices depending on parameters}, Uspehi Mat. Nauk
  \textbf{26} (1971), no.~2(158), 101--114, English Translation: Russian Math.
  Surveys, 26 (1971), 2, 29-43.

\bibitem{Baouendi}
M.~S. Baouendi and C.~Goulaouic, \emph{{Remarks on the abstract form of
  nonlinear Cauchy-Kovalevsky theorems}}, Comm. Partial Differential Equations
  \textbf{2} (1977), no.~11, 1151--1162.

\bibitem{Brjuno}
A.D. Brjuno, \emph{{Analytic form of differential equations I}}, Trans. Moscow
  Math. Soc. \textbf{25} (1971), 131--288.

\bibitem{groupes}
J.~F\'ejoz and M.D. Garay, \emph{Un th\'eor\`eme sur les actions de groupes de
  dimension infinie}, Comptes Rendus \`a l'Acad\'emie des Sciences \textbf{348}
  (2010), no.~7-8, 427--430.

\bibitem{lagrange}
M.D. Garay, \emph{{A rigidity theorem for Lagrangian deformations}}, Compositio
  Mathematica \textbf{141} (2005), no.~6, 1602--1614.

\bibitem{quantique}
\bysame, \emph{Perturbative expansions in quantum mechanics}, Annales de
  l'institut Fourier \textbf{59} (2009), no.~5, 2061--2101.

\bibitem{Oberwolfach_2012}
\bysame, \emph{{The Herman conjecture}}, Oberwolfach reports, European
  Mathematical Society, 2012.

\bibitem{Hamilton_implicit}
R.S. Hamilton, \emph{{The inverse function theorem of Nash and Moser}}, Bull.
  Amer. Math. Soc. \textbf{7} (1982), no.~1, 65--222.

\bibitem{Kolmogorov_KAM}
A.~N. Kolmogorov, \emph{On the conservation of quasi-periodic motions for a
  small perturbation of the hamiltonian function}, Dokl. Akad. Nauk SSSR
  \textbf{98} (1954), 527--530.

\bibitem{Mather_fdet}
J.~Mather, \emph{{Stability of $C^\infty$ mappings, III. Finitely determined
  map-germs}}, Publications Math\'ematiques de l'IH\'ES \textbf{35} (1968),
  127--156.

\bibitem{Moser_pde}
J.~Moser, \emph{A new technique for the construction of solutions of nonlinear
  differential equations}, Proceedings of the National Academy of Sciences of
  the U.S.A. \textbf{47} (1961), no.~11, 1824--1831.

\bibitem{Moser_KAM}
\bysame, \emph{{On the construction of almost periodic solutions for ordinary
  differential equations (Tokyo, 1969)}}, Proc. Internat. Conf. on Functional
  Analysis and Related Topics, Univ. of Tokyo Press, 1969, pp.~60--67.

\bibitem{Nagumo}
M.~Nagumo, \emph{{ \"Uber das Anfangswertproblem partieller
  Differentialgleichungen}}, Jap. J. Math. \textbf{18} (1942), 41--47.

\bibitem{Nash_imbedding}
J.~Nash, \emph{{The imbedding problem for Riemannian manifolds}}, Ann. Math.
  \textbf{63} (1965), 20--63.

\bibitem{Nirenberg}
L.~Nirenberg, \emph{{An abstract form of the nonlinear Cauchy-Kowalewski
  theorem}}, J. Differential Geometry \textbf{6} (1972), 561--576.

\bibitem{Nishida}
T.~Nishida, \emph{{A note on a theorem of Nirenberg}}, J. Differential Geom.
  \textbf{12} (1977), no.~4, 629--633.

\bibitem{Ovsyannikov}
I.V. Ovsyannikov, \emph{{A singular operator in a scale of Banach spaces}},
  Soviet Math. Dokl. \textbf{6} (1965), 1025--1028.

\bibitem{Sergeraert}
F.~Sergeraert, \emph{{Un th\'eor\`eme de fonctions implicites sur certains
  espaces de Fr\'echet et quelques applications}}, Ann. Sci. \'Ecole Norm. Sup.
  \textbf{5} (1972), no.~4, 599--660.

\bibitem{VS}
C.~Sevenheck and D.~van Straten, \emph{Deformation of singular lagrangian
  subvarieties}, {Math. Annalen} \textbf{327} (2003), no.~1, 79--102.

\bibitem{Sevryuk_reversible}
M.B. Sevryuk, \emph{{Reversible Systems}}, Lecture Notes in Math., vol. 1211,
  Springer, 1986, 319 pp.

\bibitem{Siegel_vecteurs}
C.L. Siegel, \emph{{\"Uber die Normalform analytischer Differentialgleichungen
  in der N\"ahe einer Gleichgewichtsl\"osung}}, Nach. Akad. Wiss. G\"ottingen,
  math.-phys. (1952), 21--30.

\bibitem{Stolovitch_KAM}
L.~Stolovitch, \emph{{ A KAM phenomenon for singular holomorphic vector
  fields}}, Publ. Math. Inst. Hautes \'Etudes Sci. \textbf{102} (2005),
  99--165.

\bibitem{Tyurina}
G.N. Tyurina, \emph{Locally semi-universal plane deformations of isolated
  singularities in complex space}, Math. USSR, Izv \textbf{32:3} (1968),
  967--999.

\bibitem{VanStraten_Lagrangian}
D.~van Straten, \emph{{Some problems on Lagrangian singularities}},
  Singularities and computer algebra, Kaiserslautern, October 18--20, 2004.
  (C.~Lossen and G.~Pfister, eds.), London Mathematical Society Lecture Note
  Series, vol. 324, Cambridge Univ. Press, 2006, p.~333–349.

\bibitem{Zehnder_implicit}
V.M. Zehnder, \emph{{Generalized implicit function theorems with applications
  to some small divisor problems I}}, Communications Pure Applied Mathematics
  \textbf{28} (1975), 91--140.

\end{thebibliography}
 \end{document}